\documentclass[12pt]{amsart}
\allowdisplaybreaks[1]

\usepackage[usenames]{color}

\usepackage{a4wide}
\usepackage{amssymb}
\usepackage{amsmath}
\usepackage{amscd}
\usepackage{amsfonts,mathrsfs}
\usepackage{hyperref}
\usepackage[latin1]{inputenc}

\theoremstyle{plain}
\newtheorem{lemma}{Lemma}[section]
\newtheorem{prop}[lemma]{Proposition}
\newtheorem{coro}[lemma]{Corollary}
\newtheorem{thm}[lemma]{Theorem}

\theoremstyle{definition}

\newtheorem{definition}[lemma]{Definition}
\newtheorem{remark}[lemma]{Remark}

\newcommand{\ts}{\hspace{0.5pt}}

\newcommand{\R}{\mathbb{R}\ts}

\newcommand{\IP}{\mathbb{P}}

\newcommand{\IHH}{\mathscr{H}}
\newcommand{\IAA}{\mathscr{A}}

\newcommand{\aV}[1]{\left\Vert #1\right\Vert}
\newcommand{\as}[1]{\langle #1\rangle}

\newcommand{\ow}[1]{\widetilde{#1}}

\newcommand{\Hmm}[1]{\leavevmode{\marginpar{\tiny%
$\hbox to 0mm{\hspace*{-0.5mm}$\leftarrow$\hss}%
\vcenter{\vrule depth 0.1mm height 0.1mm width \the\marginparwidth}%
\hbox to 0mm{\hss$\rightarrow$\hspace*{-0.5mm}}$\\\relax\raggedright #1}}}

\newcommand{\IC}{\mathbb{C}}
\newcommand{\IR}{\mathbb{R}}

\newcommand{\ILL}{\mathscr{L}}

\newcommand{\IN}{\mathbb{N}}
\newcommand{\IZ}{\mathbb{Z}}

\newcommand{\IX}{\mathbb{X}}

\newcommand{\dom}{\mathrm{dom}}

\newcommand{\f}{\frac}

\newcommand{\nn}{\nonumber}

\begin{document}

\title[Feynman path integrals on discrete graphs ]{Feynman path integrals for magnetic Schrödinger operators on infinite weighted graphs}

\author[B. G\"uneysu]{Batu G\"uneysu}
\address{Batu G\"uneysu, Institut f\"ur Mathematik, Humboldt-Universit\"at zu Berlin, 12489 Berlin, Germany} \email{gueneysu@math.hu-berlin.de}
\author[M. Keller]{Matthias Keller}
\address{M.~Keller:  Institut f\"ur Mathematik, Universit\"at Potsdam, 14476  Potsdam, Germany}
	\email{matthias.keller@uni-potsdam.de}

\begin{abstract} We prove a Feynman path integral formula for the \emph{unitary group} $ \exp(-itL_{v,\theta})$, $t\geq 0$, associated with a discrete magnetic Schrödinger operator $L_{v,\theta}$ on a large class of weighted infinite graphs. As a consequence, we get a new Kato-Simon estimate 
$$
|\exp(-itL_{v,\theta})(x,y)|\leq \exp(-tL_{-\mathrm{deg},0})(x,y),
$$
which controls the unitary group uniformly in the potentials in terms of a Schrödinger semigroup, where the potential $\mathrm{deg}$ is the weighted degree function of the graph.
\end{abstract}
\date{\today} %
\maketitle

\section{Introduction}

While the Schrödinger \emph{semigroup} $\exp(-tH_{v,\theta})$, $t\geq 0$, associated to an electric potential $v$ and a magnetic potential $\theta$ on the Euclidean $\IR^d$ or a general Riemannian manifold is given by a well-defined Brownian motion path integral formula \cite{simon,batu}, the Feynman-Kac-Ito formula, it is well-known that there cannot hold an analogous formula for the unitary Schrödinger group $\exp(-itH_{v,\theta})$, $t\geq 0$. For example, it can be proven \cite{RS2} that there cannot exist a complex measure $\mu$ on the space of continuous paths $[0,\infty)\to \IR^m$ such that the finite dimensional distributions of $\mu$ are given by the integral kernel $\exp(-itH_{0,0})(x,y) $ of $\exp(-itH_{0,0})$, showing that there cannot even exist\footnote{For the sake of completeness we remark that there exist several well-defined substitute results that \lq\lq{}mimick\rq\rq{} a path integral for $\exp(-itH_{v,\theta})$. For example, one can use white noise analysis \cite{hida}, an infinite dimensional distribution theory to derive such a formula in the Euclidean $\IR^d$ case, at least under some strong assumptions on the potentials.} a path integral formula in the literal sense for the unitary group of the Laplace operator $H_{0,0}=-\Delta$ in $\IR^d$. On the other hand, it is expected from some simple heuristics \cite{RS2} that the divergences of the Feynman path integral for $\exp(-itH_{v,\theta})$ actually stem from local singularities, so that in principle one can expect a well-defined Feynman path integral formula to hold true if one replaces the Riemannian manifold with an infinite weighted graph and considers the corresponding discrete magnetic Schrödinger operators thereon. \emph{The main result of this paper shows that indeed such a path integral integral formula holds true in a very general setting.} \vspace{2mm}

To this end, given a weighted graph $(X,b,m)$, possibly non-locally finite, a magnetic potential $\theta:\{b>0\}\to \IR$ and an electric potential $v:X\to\IR$, we use quadratic form methods to define a natural self-adjoint realization $L_{v,\theta}$ of the formal magnetic Schrödinger operator
$$
\ow{L}_{v,\theta}f(x) = \frac{1}{m(x)}\sum_{y \in X} b(x,y)\big (f(x)-\exp(i\theta(x,y) )f(y)\big)+v(x)f(x)
$$
in the complex Hilbert space $\ell^2(X,m)$. Operators of this type appear naturally in a gauge theoretic discretization procedure for continuum magnetic Schrödinger operators (cf. Remark \ref{pamy} below), and have been used in the tight binding approximation in solid state physics \cite{Ha}. The most prominent example of such an operator is certainly the Harper operator \cite{Ha}, whose spectral theory has been subject to the famous \lq\lq{}ten Martini problem\rq\rq{} \cite{martini}.\\
Our main result, Theorem \ref{main}, is the following Feynman path integral formula for the integral kernel $\exp(-itL_{v,\theta})(x,y)$ of $\exp(-itL_{v,\theta})$: defining a random variable
\begin{align*}
\IAA_t(v,\theta|\mathbb{X}):=i\int_0^t \theta( d \mathbb{X}_s)-i\int_{0}^{t}(v(\mathbb{X}_s)+\mathrm{deg}(\mathbb{X}_s))ds+\int^t_0\mathrm{deg}(\mathbb{X}_s)ds:\Omega\longrightarrow \IC
\end{align*}
on the space $\Omega$ of explosive $X$-valued right-continuous jump paths, one has

\begin{align}\label{fki22}
	\exp(-itL_{v,\theta})(x,y)=\frac{1}{m(y)} \int_{\{\mathbb{X}_t=y\}\cap\{N_t(\mathbb{X})<\infty\}}i^{N_t(\mathbb{X})} \exp(\IAA_t(v,\theta|\mathbb{X})) d\mathbb{P}_x,
\end{align}
where
\begin{itemize}
\item $\mathbb{P}_x$, $x\in X$, denotes the Markov family of probability measures on $\Omega$ which is induced by $L_{v,\theta}|_{v=0,\theta=0}$ by the theory of regular Dirichlet forms
\item $\mathbb{X}_t(\omega):=\omega(t)$ is the coordinate process on $\Omega$
\item $N_t(\mathbb{X})\in \IN\cup \{\infty\}$ is the number of jumps of $\mathbb{X}$ until the the time $t$
\item $\int_0^t \theta( d \mathbb{X}_s):\Omega\to\IR$ is the line integral of $\theta$ along the paths of $\mathbb{X}$
\item $\mathrm{deg}:X\to [0,\infty)$ the weighted degree function on $(X,b,m)$.
\end{itemize}

To the best of our knowledge this formula is even conceptually entirely new, in the sense that the only previously established case was $v=0$, $\theta =0$ on the unweighted standard lattice in $\IZ^d$ (cf. \cite{carmona}). The assumptions of Theorem \ref{main} are satisfied, if e.g. the electric potential $v$ is bounded from below and $\mathrm{deg}$ is bounded, noting that, however, Theorem~\ref{main} can deal with much more general situations than the latter. We expect the Feynman path integral formula (\ref{fki22}) to have several important spectral theoretic and geometric consequences: For example, by comparing (\ref{fki22}) with the usual Feynman-Kac formula 
$$
\exp(-tL_{-\mathrm{deg},0})(x,y)=\frac{1}{m(y)} \int_{\{\mathbb{X}_t=y\}\cap\{N_t(\mathbb{X})<\infty\}}  \exp\Big(\int^t_0\mathrm{deg}(\mathbb{X}_s)\Big) d\mathbb{P}_x,
$$
for the Schrödinger semigroup $\exp(-tL_{-\mathrm{deg},0})$, $t\geq 0$, one immediately gets the Kato-Simon type inequality 
$$
|\exp(-itL_{v,\theta})(x,y)|\leq \exp(-tL_{-\mathrm{deg},0})(x,y),
$$
which controls the underlying unitary magnetic Schrödinger group uniformly in both potentials in terms of the geometry of $(X,b,m)$. The latter inequality is expected to be of a fundamental importance in the context of discrete Kato-Strichartz estimates on general weighted graphs (cf. \cite{jacob} for a very recent study of such estimates for the unweighted standard lattice in $\IZ^d$). Another interesting direction could be dictated by the following observation: Given another magnetic potential $\theta\rq{}:\{b>0\}\to \IR$ and another electric potential $v\rq{}:X\to \IR$ it is straightforward to derive the following explicit Feynman path integral formula for the composition $\exp(-itL_{v,\theta})\exp(itL_{v\rq{},\theta\rq{}})$

\begin{align*}
&\big[\exp(-itL_{v,\theta})\exp(itL_{v',\theta'}) \big](x,y)\\
&=\int_{ \{ \omega:N_t(\mathbb{X}(\omega))<\infty\}}i^{N_t(\mathbb{X}(\omega))} \exp\big(\IAA_t(v,\theta|\mathbb{X}(\omega))\big) m(\omega(t))^{-1}\Psi_t(v\rq{},\theta\rq{},y,\omega) d\IP^x(\omega),
\end{align*} 
where the random variable $\Psi_t(v\rq{},\theta\rq{},y,\cdot):\Omega\to\IC $ is given by
$$
\Psi_t(v\rq{},\theta\rq{},y,\omega ):=\int_{\{\omega\rq{}:\omega'(t)=\omega(t)\}\cap\{\omega\rq{}:N_t(\mathbb{X}(\omega\rq{}))<\infty\}}\overline{i^{N_t(\mathbb{X}')} \exp\big(\IAA_t(v',\theta'|\mathbb{X}(\omega\rq{}))\big)}d\IP^y(\omega\rq{}).
$$
 We believe that this result, which is again conceptually completely new, will turn out to be very useful in the context of scattering theory (cf. \cite{jacob} for scattering theory results on the unweighted standard lattice in $\IZ^d$).

\section{Main results}

\subsection{Weighted graphs}\label{s:graphs} Let $b$ be a graph over the countable set $ X $, i.e., 
$$
 b:X\times X \longrightarrow  [0,\infty) \text{ is symmetric with  $b(x,x)=0$ and $\sum_{y \in X}b(x,y) < \infty$ for all $x \in X$.}
$$

Then, the elements of $X$ are called the \emph{vertices} of $(X,b)$ and all $(x,y) \in X\times X$ with $b(x,y) >0$ are called the \emph{edges} of $(X,b)$, where given $x\in X$ every $y\in X$ with $b(x,y)>0$ is called a \emph{neighbor of $x$} and we write $ x\sim y $. The graph $(X,b)$ is called \emph{locally finite}, if every vertex has only a finite number of neighbors. Furthermore, a \emph{path} on the graph $(X,b)$ is a (finite or infinite) sequence of pairwise distinct vertices $(x_{j})$ such that $x_j \sim x_{j+1}$ for all $j$, and $X$ is called \emph{connected}, if for any $x,y\in X$ there is a finite path $(x_j)^n_{j=0}$ such that $x_0=x$ and $x_n=y$. We equip $X$ with the discrete topology, so that any function $m: X \to (0,\infty)$ gives rise to a Radon measure of full support on $X$ by setting $m(A) := \sum_{x \in A}m(x)$. Then, we say $ b $ is a graph over the measure space $(X,m)$ and call the triple $(X,b,m)$ a \emph{weighted graph}. For $x\in X$, we denote the \emph{weighted degree function} by
$$
\text{deg}(x):= \frac{1}{m(x)}\sum_{y \in X} b(x,y).
$$

\subsection{Self-adjoint realizations of magnetic Schrödinger operators}

In the following, we understand all spaces of functions to be complex-valued, and $i:=\sqrt{-1}$. Let ${C}(X)$ be the linear space of functions on $X$ and ${C}_c(X)$ its subspace of functions with finite support. We denote the standard scalar product and norm on $\ell^2(X,m)$ with $\as{\cdot,\cdot}$ and $\aV{\cdot}$. A \emph{magnetic potential} on $(X,b)$ is an antisymmetric function 
$$
\theta: \{b>0\}\longrightarrow \IR\quad\mbox{such that}\quad\theta(x,y) = - \theta(y,x),\;x,y\in X.
$$ 
Any function $v: X \to \R$ will be simply called a \emph{electric potential} on $X$.
\medskip

We define a symmetric densely defined sesqui-linear form in the complex Hilbert space $\ell^2(X,m)$ with domain of definition ${C}_c(X)$ by
\begin{align}\label{dert}
Q_{v,\theta}^{(c)}(f,g):= & \frac{1}{2}\sum_{x\sim y} b(x,y)\Big(f(x)-\exp(i \theta(x,y) )f(y)\Big)\overline{\Big(g(x)-\exp(i \theta(x,y)) g(y)\Big)}\\
&\quad+\sum_{x\in X}v(x)f(x)\overline{g(x)}m(x)\nn.
\end{align}

\begin{remark}\label{pamy} Although not obvious, the above definition of $Q_{v,\theta}^{(c)}$ actually reflects a natural discretization procedure. To see that, one has to take the $U(1)$ gauge theory behind magnetic Schrödinger operators in $\IR^d$ into account: Assume $\tilde{\theta}$ is a $C^1$ and real-valued $1$-form on $\IR^d$ (= a magnetic potential) and $\tilde{v}:\IR^d\to \IR$ is continuous (= an electric potential). Then $\tilde{\theta}$ induces the metric covariant derivative $\nabla^{\tilde{\theta}}:= d+i\tilde{\theta}$ on the trivial complex line bundle $\IR^d\times \IC\to \IR^d$ over $\IR^d$, and one can define a symmetric sesquilinear form in $L^2(\IR^d)$ by
\begin{align}\label{pawy}
Q^{(c)}_{\tilde{v},\tilde{\theta}}= \frac{1}{2}\int \sum^{d}_{j=1}\nabla^{\tilde{\theta}}_{\partial_j}f\cdot\overline{\nabla^{\tilde{\theta}}_{\partial_j}g } \ dx+\int \tilde{v}f \overline{g} \ dx, \quad f,g\in C^{\infty}_c(\IR^d).
\end{align}
The starting point for a discretization of the above sesquilinear form is simply to drop the limit and to set $\delta=1$ in the formula (cf.  formula (7.66) in \cite{thalmaier})
$$
\nabla^{\tilde{\theta}}_{\pm\partial_j}f(x)=\lim_{\delta\to 0}\f{1}{\delta}\exp\Big(i\int_{\gamma^{\delta}_{x,x\pm e_j}}\tilde{\theta}\Big)f(x\pm e_j)-\f{1}{\delta}f(x), \quad f \in C^{\infty}(\IR^d), 
$$
where 
$$
\gamma^{\delta}_{x,x\pm e_j}:[0,\delta]\longrightarrow \IR^d, \quad \gamma^{\delta}_{x,x\pm e_j}(t):=\f{1}{\delta}(\delta-t)x+\f{t}{\delta}(x\pm e_j)
$$
is the straight line which starts from $x$ and ends in $x\pm e_j$ at the time $\delta>0$. Note that $\exp\left(i\int_{\gamma^{\delta}_{x,x\pm e_j}}\tilde{\theta}\right)$ is precisely the (inverse) parallel transport along $\gamma^{\delta}_{x,x\pm e_j}$ with respect to the $U(1)$ covariant derivative $\nabla^{\tilde{\theta}}$. Then $\theta(x,y):=\int_{\gamma^1_{x,y}}\tilde{\theta}$ defines a magnetic potential on $\IZ^d$ with its standard unweighted graph structure $b_{\IZ^d}(x,y)=1$ if $|x-y|_{\IR^d}=1$, and $b_{\IZ^d}(x,y)= 0$ else. Note that $b_{\IZ^d}(x,y)>0$ if and only of $y$ is of the form $x\pm e_j$. With
\begin{align*}
 &\nabla^{\theta} f(x,y):= \exp\left(i\theta(x,y)\right)f(y)-f(x), \quad (x,y)\in \{b_{\IZ^d}>0\},\\
&  v(x):=\tilde{v}(x), \quad x\in \IZ^d,
\end{align*}
we arrive at the sesquilinear form
\begin{align*}
Q_{v,\theta}^{(c)}(f,g):=&\frac{1}{2}\sum_{x\in X}\sum_{y:|x-y|_{\IR^d}=1} \nabla^{\theta} f(x,y) \overline{\nabla^{\theta} g(x,y)}+\sum_{x\in X}v(x)f(x)\overline{g(x)}\\
&= \frac{1}{2}\sum_{  |x-y|_{\IR^d}=1}  \Big(f(x)-\exp(i \theta(x,y) )f(y)\Big)\overline{\Big(g(x)-\exp(i \theta(x,y)) g(y)\Big)}\\
&\quad+\sum_{x\in X}v(x)f(x)\overline{g(x)}
\end{align*}
in $\ell^2(\IZ^d)$ which is precisely of the type (\ref{dert}) for $X=\IZ^d$, $b=b_{\IZ^d}$, $m\equiv 1$, and in addition formally of the type (\ref{pawy}). The above discretization procedure could be summarized as follows: one replaces the covariant derivative (an infinitesimal object) by its parallel transport. Field theoretic variants of this procedure are standard in lattice gauge theory\footnote{The authors would like to thank Burkhard Eden and Matthias Staudacher in this context.}.
\end{remark}
\bigskip

After discussing how the form $ Q^{(c)}_{v,\theta} $ arises from a discretization procedure, we continue by introducing the associated formal operator. Let
$$
\widetilde{C}(X):= \Big\{f\in {C}(X):\sum_{y\in X} b(x,y)|f(y)| < \infty \text{ for all }x\in X\Big\},
$$
and we define the formal difference operator $\ow{L}_{v,\theta}:\widetilde{C}(X) \to {C}(X)$ by
$$
\ow{L}_{v,\theta}f(x) = \frac{1}{m(x)}\sum_{y \in X} b(x,y)\big (f(x)-\exp(i\theta(x,y) )f(y)\big)+v(x)f(x).
$$
The form $Q_{v,\theta}^{(c)}$ and the operator $\ow L_{v,\theta}$ are related by Green's formula: for all $f\in \widetilde{C}(X)$, $g \in {C}_c(X)$, one has
\begin{align*}
\sum_{x\in X}&\ow{L}_{v,\theta}f(x)\overline{g(x)}m(x) = \sum_{x\in X} f(x)\overline{\ow{L}_{v,\theta}g(x)}m(x)\\
&=\frac{1}{2}\sum_{x,y \in X} b(x,y) \Big(f(x)-\exp(i \theta(x,y) )f(y)\Big)\overline{\Big(g(x)-\exp(i \theta(x,y))g(y)\Big)} \\
&\quad+ \sum_{x\in X}v(x) f(x)\overline{g(x)} m(x).\nn
\end{align*}
Moreover, if $\ow L_{v,\theta}[{C}_c(X)]\subseteq \ell^{2}(X,m)$, then for all $f,g\in C_{c}(X)$ one has
\begin{align*}
    Q^{(c)}_{v,\theta}(f,g)=\langle \ow L_{v,\theta} f, g\rangle=\langle f,\ow L_{v,\theta} g\rangle.
\end{align*}

If $Q^{(c)}_{v,\theta}$ is bounded from below and closable, we denote its closure by $Q_{v,\theta}$ and the corresponding self-adjoint operator by $L_{v,\theta}$, referred to as the \emph{magnetic Schrödinger operator induced by $(\theta,v)$}. From the Green's formula it is obvious that $ L_{v,\theta} $ is a restriction of $ \ow L_{v,\theta}  $, i.e.,
\begin{align*}
 L_{v,\theta} =\ow L_{v,\theta} \mbox{ on } \dom(L_{v,\theta} ).
\end{align*}  Likewise, the strongly continuous unitary group of operators
\begin{align*}
\exp(-itL_{v,\theta})\in \ILL(\ell^2(X,m)),  \quad {t\in\IR},
\end{align*}
defined by the spectral calculus, is called \emph{magnetic Schrödinger group} induced by $\theta$ and $v$. The importance of this group for quantum mechanics is that for every $\psi\in \dom(L_{v,\theta})$ the function $t\mapsto \psi(t):=\exp(-itL_{v,\theta})\psi$ is the unique strong $C^1$-map $\IR\to \ell^2(X,m)$ which satisfies the Schrödinger equation
$$
(d/dt) \psi (t) = -iL_{v,\theta}\psi(t),\quad \psi(0)=\psi. 
$$

The following remark (cf. Lemma 2.3, Lemma 2.11 and Theorem 2.12 in \cite{GKS}) addresses some functional analytic subtleties of these operators:

\begin{remark} 1. If $X$ is locally finite, then one has $\widetilde{C}(X)={C}(X)$, however, in general, $\widetilde{C}(X)$ does not include $\ell^{2}(X,m)$.\\
2. The condition $\ow L_{v,\theta}[{C}_c(X)]\subseteq \ell^{2}(X,m)$ for some (or equivalently all) $( v,\theta)$ is equivalent to
$$
\sum_{y\in X}\frac{ b(x,y)^2}{m(y)}<\infty\quad\text{ for all $x\in X$}.
$$
We refer the reader to \cite{GKS, Mi1,Mi2} for essential self-adjointness results under the assumption $\ow L_{v,\theta}[{C}_c(X)]\subseteq \ell^{2}(X,m)$.\\
3. If  $Q^{(c)}_{v,0}$ is bounded from below then $Q^{(c)}_{v,\theta}$ is automatically closable for all magnetic potentials $ \theta $. If  $\ow L_{v,\theta}[{C}_c(X)]\subseteq \ell^{2}(X,m)$ then $\ow L_{v,\theta}$ is a symmetric operator on $ C_{c}(X)\subseteq \ell^{2}(X,m) $ and, hence, if $ Q_{v,\theta}^{(c)} $ is bounded below, then  $ Q_{v,\theta}^{(c)} $ is closable and $ L_{v,\theta} $ is a restriction of $\ow L_{v,\theta}$.   
\end{remark} 




\section{Stochastic processes on weighted graphs}\label{proc}

Given a Hausdorff space $Y$ we denote its Alexandrov compactification by $\hat{Y}=Y\cup \{\infty_Y\}$ if $Y$ is noncompact and locally compact and $\hat{Y}=Y$ if $Y$ is compact. Let us introduce the probabilistic framework: Let us denote with $\Omega$ the space of right-continuous paths $\omega:[0,\infty)\to \hat{X}$ having left limits, which is equipped with its Borel-sigma-algebra $\mathcal{F}$. The latter is filtered by the filtration $\mathcal{F}_*$ generated by the coordinate process
$$
\mathbb{X}: [0,\infty)\times \Omega \longrightarrow \hat{X},\quad \mathbb{X}_t(\omega):=\omega(t).
$$
For every subset $W\subset X$, let
\[
\tau_W:= \inf\{s\geq 0: \ \mathbb{X}_s\in X\setminus W\}:\Omega\longrightarrow [0,\infty]
\]
be the first exit time of $\mathbb{X}$ from $W$. Note that
$$
\{t<\tau_W\}=\{\mathbb{X}_s\in W\text{ for all $s\in [0,t]$}\}\quad\text{ for all $t\geq 0$.}
$$

For the sake of brevity we write
$$
Q^{(c)}:=Q^{(c)}_{v,\theta}|_{v=0,\theta=0},\quad Q:=Q_{v,\theta}|_{v=0,\theta=0} ,\quad L:= L_{v,\theta}|_{v=0,\theta=0}
$$
for the underlying free forms and operator, respectively. An essential property of $Q$ is that it is a regular symmetric Dirichlet form in $\ell^2(X,m)$. It follows automatically from Fukushima's theory that for every $x\in X$ there exists a unique probability measure $\mathbb{P}_x$ on $(\Omega,\mathcal{F})$ such that for all finite sequences $0=t_0<t_1<\dots < t_l$ and all $x=x_0,x_1,\dots, x_l\in \hat{X}$ one has
$$
\mathbb{P}_x\{\mathbb{X}_{t_1}=x_1,\dots, \mathbb{X}_{t_l}=x_l\}= \exp(-\delta_0L)(x_{0},x_1)m(x_1) \cdots \exp(-\delta_{l-1}L)(x_{l-1},x_l)m(x_l),
$$
where $\delta_j:=t_{j+1}-t_{j}$, and where $\exp(- tL)(\bullet,\bullet)$ is extended to $\hat{X}\times \hat{X}$ according to
\begin{align*}
&\exp(- tL)(y,\infty_X):=0,\quad  \exp(- tL)(\infty_X,\infty_X)=1,\\
&  \exp(- tL)(\infty_X,y):=  1-\sum_{z\in X} \exp(- tL)(z,y)m(z),\quad  y\in X .
\end{align*}

In addition, Fukushima's result entails that these measures are concentrated on paths having $\infty_X$ as a cemetery,
\begin{align}
\label{inti}\mathbb{P}_x\Big( \{\tau_X=\infty\}\cup \{ \text{$\tau_X<\infty$ and $\mathbb{X}_t=\infty_X$ for all $t\in [\tau_X,\infty)$} \} \Big)=1,
\end{align}
and that 
$$
\mathscr{M}:=(\Omega,\mathcal{F},\mathcal{F}_*,\mathbb{X},(\mathbb{P}_x)_{x\in X})
$$
is a reversible strong Markov process. For every $n\in\IN_{\geq 0}$ let $\tau_n:\Omega \to [0,\infty]$ denote the $n$-th jump time of $\mathbb{X}$ (with $\tau_0:=0$), an $\mathcal{F}_*$-stopping time. Let 
$$
N(\mathbb{X}):[0,\infty)\times \Omega\longrightarrow  \hat{\IN},\quad  N_t(\mathbb{X}):=\text{ number of jumps of $\IX|_{[0,t]}$} ,
$$
an $\mathcal{F}_*$-adapted process. We then define 
$$
\tau:=\lim_{n\to\infty}\tau_n:\Omega \longrightarrow [0,\infty],
$$
another $\mathcal{F}_*$-stopping time.

\begin{lemma}\label{error} a) For all $t\geq 0$, $x,y\in X$ one has 
\begin{align}
\label{inti2}&\mathbb{P}_x\{1_{\{N_t(\mathbb{X})<\infty\} }=1_{\{t<\tau_X\}}\} =1,\\
\label{inti3}& \mathbb{P}_x\left({\{N_t(\mathbb{X}) = 0\}} \right)=\exp(-t\mathrm{deg}(x)),\\
\label{inti4}&\mathbb{P}_x\{ \mathbb{X}_{\tau_n}\sim\mathbb{X}_{\tau_{n+1}}\text{ \emph{for all $n\in\IN$}}\}=1,\\
\label{inti5}&\mathbb{P}_x(N_t(\mathbb{X}) = 1,\mathbb{X}_{\tau_1}= y)/t\to b(x,y)/m(x)\quad\text{  as $t\searrow 0$.}
\end{align}
b) Let $f \in {C}_c(X)$, $t>0$, and let the function $\varphi_{t,f}:X\to \IC$ be defined by
$$
\varphi_{t,f}(x) := \frac{1}{t}\mathbb{E}_x\left[1_{\{2\leq N_t(\mathbb{X})<\infty\}} f(\mathbb{X}_t) \right]. 
$$
Then, for all $x\in X$, one has $\varphi_{t,f}(x)\to 0$ as $t\searrow 0$.
\end{lemma}

\begin{proof} a) To see (\ref{inti2}), note that the inclusion $ \{N_{t}<\infty\}\subset  \{t<\tau_{X}\} $ $\mathbb{P}_x$-a.s., is immediate since the form $Q$ does not have a killing term. Using now the formula
$$
\exp(-tL)f(x)=\mathbb{E}_x\left[1_{\{t<\tau_{X}\}} f(\mathbb{X}_t) \right],
$$	
which has been shown in \cite{GKS}, we get
$$
0=\mathbb{E}_x\left[(1_{\{t<\tau_{X}\}}-1_{\{N_t(\mathbb{X})<\infty\}} )f(\mathbb{X}_t) \right],
$$
for all $f\in \ell^2(X,m)$. Letting $f$ tend to $1$ from below and using
$$
1_{\{t<\tau_{X}\}}-1_{\{N_t(\mathbb{X})<\infty\}}\geq 0\quad\text{$\mathbb{P}_x$-a.s.},
$$
we arrive at 
$$
0=\mathbb{E}_x\left[1_{\{t<\tau_{X}\}}-1_{\{N_t(\mathbb{X})<\infty\}}  \right]=\mathbb{E}_x\left[\left|1_{\{t<\tau_{X}\}}-1_{\{N_t(\mathbb{X})<\infty\}}  \right|\right],
$$
so that
$$
1_{\{t<\tau_{X}\}}=1_{\{N_t(\mathbb{X})<\infty\}}\quad\text{$\mathbb{P}_x$-a.s.}.
$$
The properties (\ref{inti3}), (\ref{inti4}), (\ref{inti5}) and b) have been shown in \cite{GKS}. 
\end{proof}

Given a magnetic potential $\theta$ on $(X,b)$, the stochastic line integral of $\mathbb{X}$ along $\theta$ is defined by
\begin{align*}
&\int_0^{\bullet} \theta(  d \mathbb{X}_s): [0,\infty)\times \Omega\longrightarrow \IR\quad \int_0^{t} \theta(  d \mathbb{X}_s)
:=\sum_{n= 1}^{N_t(\mathbb{X})}\theta(\mathbb{X}_{\tau_{n-1}}, \mathbb{X}_{\tau_{n}}),
\end{align*}
where $\int_0^{t} \theta(  d \mathbb{X}_s)$ is set $0$ if $N_t(\mathbb{X})=0$, or if $N_t(\mathbb{X})=\infty$, or if $1\leq N_t(\mathbb{X})<\infty$ and $b(\mathbb{X}_{\tau_{n-1}},\mathbb{X}_{\tau_{n}})=0$ for some $n=1,\dots,N_t(\mathbb{X})$. For an electric potential $v$ on $X$ we also have the usual Riemannian integral
\begin{align*}
&\int_0^{\bullet} v( \mathbb{X}_s) ds: [0,\infty)\times \Omega\longrightarrow \IR,\quad \int_0^{t} v (\mathbb{X}_s)ds
=\sum_{n= 1}^{N_t(\mathbb{X})+1} v(\mathbb{X}_{\tau_{n-1}})(\tau_{n}-\tau_{n-1}),
\end{align*}
where $\int_0^{t} v (\mathbb{X}_s)ds$ is set $0$, if $N_t(\mathbb{X})=\infty$. Clearly these processes are $\mathcal{F}_*$-adapted. 


\section{The Feynman path integral formula}\label{haup}
\subsection{Statement}

We recall that by the countability of $X$, for every bounded operator $A$ in $\ell^2(X,m)$ there is a uniquely determined map
$$
A(\cdot,\cdot): X\times X\longrightarrow \IC
$$
which satisfies 
$$
Af(x)=\sum_{y\in X}A(x,y)f(y)m(y)\quad\text{for all $f\in \ell^2(X,m)$, $x\in X$}.
$$
In fact, with $\delta_x\in C_c(X)$ the usual delta-function centered at $x$, one has
\begin{align}\label{l2}
 A(x,y)=m(x)^{-1}\overline{A^*\delta_x(y)},\quad\text{ so that $A(x,\cdot)\in \ell^2(X,m)$ for all $x$}.
\end{align}
 In addition, it holds that
\begin{align}\label{l3}
A^*(x,y)=\overline{A(y,x)},\quad\text{ so that also $A(\cdot,x)\in \ell^2(X,m)$ for all $x$,}
\end{align}
and one has the composition formula
\begin{align}\label{compo}
[AB](x,y)=\sum_{z\in X} A(x,z)B(z,y) m(z)
\end{align}
for the integral kernel of a composition.\vspace{1mm}

Given a magnetic potential $\theta$ on $(X,b)$ and an electric potential $v$ on $X$ we define an $\mathcal{F}_*$-adapted process
$$
\IAA(v,\theta|\mathbb{X}):\Omega\times [0,\infty)\longrightarrow \IC
$$
by setting 
\begin{align*}
\IAA_t(v,\theta|\mathbb{X}):=i\int_0^t \theta( d \mathbb{X}_s)-i\int_{0}^{t}(v(\mathbb{X}_s)+\mathrm{deg}(\mathbb{X}_s))ds+\int^t_0\mathrm{deg}(\mathbb{X}_s)ds,\quad t\geq 0.
\end{align*}

Here comes our main result:

\begin{thm}\label{main}\emph{(Feynman path integral formula)} Let $\theta$ be a magnetic potential on $(X,b)$ and let $v$ be an electric potential on $X$ such that $Q^{(c)}_{v,\theta}$ and $Q^{(c)}_{-\mathrm{deg},0}$ are semi-bounded from below and closable. Then for all $t\geq 0$, $x,y\in X$, one has
	\begin{align}\label{fki}
	\exp(-itL_{v,\theta})(x,y) = \f{1}{m(y)}\mathbb{E}_x\left[1_{\{\mathbb{X}_t=y\}\cap\{N_t(\mathbb{X})<\infty\}}i^{N_t(\mathbb{X})}\exp(\IAA_t(v,\theta|\mathbb{X}))\right].
	\end{align}
	\end{thm}

The proof of Theorem \ref{main} is given in the next two sections. One first proves the Feynman path integral formula on finite subgraphs, and then uses an exhaustion argument which relies on a new result from Mosco convergence theory (cf. Theorem \ref{mosco.char}) which is proved in the appendix. \vspace{2mm}

Note that by (\ref{inti}) and Lemma \ref{error}, for all $x,y\in X$, $t\geq 0$, we actually have 
\begin{align*}
1_{\{\mathbb{X}_t=y\}\cap\{N_t(\mathbb{X})<\infty\}}&=1_{\{\mathbb{X}_t=y\}}=1_{\{\mathbb{X}_t=y\}\cap\{N_t(\mathbb{X})<\infty\} \cap \{t<\tau_X\}} \\
&=1_{\{\mathbb{X}_t=y\}\cap \{N_t(\mathbb{X})<\infty,       
  \>b(\mathbb{X}_{\tau_{n-1}},\mathbb{X}_{\tau_{n}})>0\text{ for all $n=0,\dots,N_t(\mathbb{X})$}  \} \cap \{t<\tau_X\}},\quad\text{$\mathbb{P}_x$-a.s. }
\end{align*}


\begin{remark}Clearly, $Q^{(c)}_{-\mathrm{deg},0}$ is closable and bounded from below whenever $ \deg $ is  a bounded function (in which case the form $Q^{(c)}_{-\mathrm{deg},0}$ is bounded as a sum of two bounded forms, namely,  
$Q^{(c)}_{0,0}$ and the form induced by $ -\deg $, cf. \cite{HKLW} for the boundedness of $Q^{(c)}_{0,0}$). The boundedness of $ \deg $ also implies the stochastic completess of $(X,b,m)$. The boundedness of  $ \deg $ is certainly a natural assumption in the context of solid state physics.\\
In the case of locally finite graphs the operator $ \widetilde{L}_{-\deg,0} $ is a well-defined symmetric operator with domain of definition $ C_{c}(X) $, and is therefore closable. The semiboundedness of $\widetilde{L}_{-\deg,0}$ has been investigated by 
	Gol\'enia  in \cite{Gol}, where the author examines whether the weighted adjacency matrix of a locally finite graph is unbounded from below. Indeed, the quadratic form of the adjacency matrix acting on the finitely supported functions is exactly $Q^{(c)}_{-\mathrm{deg},0}$. A negative result in this context is that, in case the edge weights $ b $ are unbounded, it follows that $Q^{(c)}_{-\mathrm{deg},0}$ is unbounded from below.
\end{remark}

As it should be, the Feynman path integral formula also holds for negative times in the following sense: For all $t\geq 0$, $x,y\in X$, we have, using (\ref{l3}), 
\begin{align*}
&\exp(-i(-t)L_{v,\theta})(x,y)=\exp(itL_{v,\theta})(x,y)=\overline{\exp(-itL_{v,\theta})(y,x)}\\
&=\f{1}{m(x)}\mathbb{E}_y\left[1_{\{\mathbb{X}_{t}=x\}\cap \{N_t(\mathbb{X})<\infty\}}\overline{i^{N_t(\mathbb{X})} \exp(\IAA_t(v,\theta|\mathbb{X}))} \right].
\end{align*}
The above result for magnetic Schrödinger groups should be compared with the Feynman-Kac-Ito formula \cite{GKS} for magnetic Schrödinger semigroups: The latter states that if $Q^{(c)}_{v,\theta}$ and $Q^{(c)}_{v,0}$ are bounded from below and closable, then for all $t\geq 0$, $x,y\in X$ one has
\begin{align*} 
	\exp(-tL_{v,\theta})(x,y) = \f{1}{m(y)}\mathbb{E}_x\left[1_{\{\mathbb{X}_t=y\}} \exp\left(i\int_0^t \theta( d \mathbb{X}_s)- \int_{0}^{t} v(\mathbb{X}_s) ds\right)\right].
\end{align*}
It would also be interesting to see to what extent the Feynman path integral formula can be generalized to the setting of covariant Schrödinger operators in weighted graphs (a generalized Feynman-Kac-Ito formula \cite{GKS} for covariant Schrödinger semigroups has been established in \cite{GMT} and used in \cite{semic} to calculate semiclassical limits).\vspace{2mm}

An immediate but nevertheless important consequence of the above formulae is the following very surprising Kato-Simon  domination for magnetic Schrödinger groups in terms of certain geometric Schrödinger semigroups:

\begin{coro}\label{apisss} Under the assumptions of Theorem \ref{main}, for all $t\geq 0$, $x,y\in X$ one has
$$
|\exp(-itL_{v,\theta})(x,y)|\leq \exp(-tL_{-\mathrm{deg},0})(x,y).
$$
\end{coro}

This estimate is to be compared with the domination result for magnetic Schrödinger semigroups, which reads
$$
|\exp(- tL_{v,\theta})(x,y)|\leq \exp(-tL_{v,0})(x,y),
$$
provided $Q^{(c)}_{v,\theta}$ and $Q^{(c)}_{v,0}$ are semi-bounded from below and closable. We expect that Corollary \ref{apisss} should play an important role in the derivation of Kato-Strichartz estimates on general weighted graphs (cf. Theorem 1.1 in \cite{jacob} for the unweighted standard lattice in $\IZ^d$.)\\
As a byproduct of our proof of the Feynman path integral formula, we also get the following result for possibly infinite subgraphs: To this end, if $W\subset X$ is any possibly infinite subset, we define $Q^{(c,W)}_{v,\theta}$ to be the restriction of $Q^{(c)}_{v,\theta}$ to $ C_{c}(W) $. Then, taking the closure in 
$$
\ell^2(W,m):=\ell^2(W,m|_W)
$$
yields a closed form 
$Q^{(W)}_{v,\theta}$ with associated operator $L^{(W)}_{v,\theta}$.

\begin{coro} Under the assumptions of Theorem \ref{main}, let $W\subset X$ be an arbitrary subset. Then for all $t\geq 0$, $x,y\in W$ one has the following Feynman path integral formula,
	\begin{align*} 
	\exp(-itL^{(W)}_{v,\theta})(x,y)=  \f{1}{m(y)}\mathbb{E}_x\left[1_{\{\mathbb{X}_t=y\}\cap \{N_t(\mathbb{X})<\infty\}\cap \{t<\tau_W\}}i^{N_t(\mathbb{X})} \exp(\IAA_t(v,\theta|\mathbb{X}))\right].
	\end{align*}  
\end{coro}

In view of the definition of the wave operators from time dependent scattering theory, we expect that the following formula will play an important role in the context of scattering theory (cf. \cite{jacob} for some scattering results on the unweighted standard lattice in $\IZ^d$):

\begin{prop}\label{scatt} Let $\theta$, $\theta'$ be magnetic potentials on $(X,b)$ and let $v,v'$ be electric potentials on $X$ and assume that $Q^{(c)}_{v,\theta}$, $Q^{(c)}_{v',\theta'}$ and $Q^{(c)}_{-\mathrm{deg},0}$ are semi-bounded from below and closable. Then for all $t\geq 0$, $x,y\in X$, the integral kernel of $\exp(-itL_{v,\theta})\exp(itL_{v',\theta'})$ is given by
\begin{align*}
&\big[\exp(-itL_{v,\theta})\exp(itL_{v',\theta'}) \big](x,y)\\
&= \mathbb{E}_x\left[1_{\{  N_t(\mathbb{X})<\infty\}}i^{N_t(\mathbb{X})} \IAA_t(v,\theta|\mathbb{X})m(\IX_t)^{-1}\mathbb{E}_y\left[1_{\{\mathbb{X}'_t=\mathbb{X}_t\}\cap\{N_t(\mathbb{X}')<\infty\}}\overline{i^{N_t(\mathbb{X}')} \exp\big(\IAA_t(v',\theta'|\mathbb{X}')\big)}\right]\right],
\end{align*}
where $\mathbb{X}'$ denotes an independent copy of $\mathbb{X}$.
\end{prop}

Note that, explicitly, the formula from Proposition \ref{scatt}  reads as follows:
\begin{align*}
&\big[\exp(-itL_{v,\theta})\exp(itL_{v',\theta'}) \big](x,y)\\
&=\int_{ \{ \omega:N_t(\mathbb{X}(\omega))<\infty\}}i^{N_t(\mathbb{X}(\omega))} \exp\big(\IAA_t(v,\theta|\mathbb{X}(\omega))\big) m(\omega(t))^{-1}\\
&\quad\times\int_{\{\omega\rq{}:\>\omega'(t)=\omega(t)\}\cap\{\omega\rq{}:N_t(\omega'(t))<\infty\}}\overline{i^{N_t(\mathbb{X}(\omega\rq{}))} \exp\big(\IAA_t(v',\theta'|\mathbb{X}(\omega\rq{}))\big)}d\IP^y(\omega\rq{})\>\> d\IP^x(\omega).
\end{align*}

\begin{proof}[Proof of Proposition \ref{scatt}] Using the composition formula (\ref{compo}) and setting  
$$
h(z):=\exp(itL_{v',\theta'})(z,y)=m(z)^{-1}\mathbb{E}_y\left[1_{\{\mathbb{X}'_t=z\}\cap\{N_t(\mathbb{X})<\infty\}}\overline{i^{N_t(\mathbb{X}')} \exp(\IAA_t(v',\theta'|\mathbb{X}'))}\right].
$$
for fixed $y$, we have $h\in \ell^2(X,m)$ by (\ref{l3}) and 
\begin{align*}
\big[\exp(-itL_{v,\theta})\exp(itL_{v',\theta'}) \big](x,y)
&=\exp(-itL_{v,\theta})h(x)\\
&=\mathbb{E}_x\left[1_{\{N_t(\mathbb{X})<\infty\}}i^{N_t(\mathbb{X})} \exp(\IAA_t(v,\theta|\mathbb{X}))h(\mathbb{X}_t)\right]\\
&=\mathbb{E}_x\left[1_{\{  N_t(\mathbb{X})<\infty\}}i^{N_t(\mathbb{X})} \exp(\IAA_t(v,\theta|\mathbb{X}))m(\IX_t)^{-1}\right.\\
&\quad\times\left.\mathbb{E}_y\left[1_{\{\mathbb{X}'_t=\mathbb{X}_t\}\cap\{N_t(\mathbb{X}')<\infty\}}\overline{i^{N_t(\mathbb{X}')} \exp\big(\IAA_t(v',\theta'|\mathbb{X}')\big)}\right]\right],
\end{align*}
completing the proof.

\end{proof}


\subsection{Proof of the Feynman path integral formula for finite subgraphs}

Let $\theta$ be a magnetic potential and $v$ be an electric potential on $X$.

\begin{prop}\label{p:finite} Let $W\subseteq X$ be finite. Then for all $f\in \ell^{2}(W,m)$, $x\in W$, $t\ge0$, one has
\begin{align*}
 \exp(-itL^{(W)}_{v,\theta})f(x)=\mathbb{E}_x\left[1_{\{t<\tau_W\}}i^{N_t(\mathbb{X})}   \exp(\IAA_t(v,\theta|\mathbb{X}))f(\mathbb{X}_t)\right].
\end{align*}
\end{prop}

The proof of the proposition above is based on three auxiliary lemmas.

\begin{lemma}\label{l:semigroup} Let $W \subseteq X$ be finite. Then, $(U_t(v,\theta,W))_{t\ge0}$
defined for $f\in\ell^{2}(W,m)$ by
\begin{align*}
U_t(v,\theta,W)f(x):=\mathbb{E}_x\left[1_{\{t<\tau_W\}} i^{N_t(\mathbb{X})}\exp(\IAA_t(v,\theta|\mathbb{X})) f(\mathbb{X}_t)\right],\quad x\in W,t\ge0,
\end{align*}
is a strongly continuous semigroup of bounded operators on $\ell^{2}(W,m)$.
\end{lemma} 

\begin{proof} The asserted boundedness is trivial and the semigroup property follows from the strong Markov property of $\mathbb{X}$. By the semigroup property it is enough to check strong continuity at $t=0$, which can be easily checked using the boundedness of the integrand and the right continuity of $\mathbb{X}$.
\end{proof}

%
%


\begin{lemma}\label{generator} Let $W\subseteq X$ be finite. Then, for all $f\in \ell^2(W,m)$ and $x\in W$, one has
\begin{align*}
    \lim_{t\searrow0}\frac{U_t(v,\theta,W)f(x)-f(x)}{t}=-i L^{(W)}_{v,\theta}f(x).
\end{align*}
\end{lemma}
\begin{proof}
We fix an arbitrary $x\in W$ and compute
\begin{align*}
\lefteqn{\frac{U_t(v,\theta,W)f(x)-f(x)}{t} } \nonumber \\&=\frac{\mathbb{E}_x\left[1_{\{N_t(\mathbb{X}) = 0\}}\exp(\IAA_t(v,\theta|\mathbb{X})) f(x)\right] -f(x)}{t}
&+\frac{\mathbb{E}_x\left[1_{\{N_t(\mathbb{X}) = 1,\mathbb{X}_{\tau_{1}}\in W\}}i \exp(\IAA_t(v,\theta|\mathbb{X})) f(\mathbb{X}_t) \right]}{t} \\
&\quad + \psi_t(x)
\end{align*}
The error term $\psi_t(x)$ satisfies $|\psi_t(x)| \leq \varphi_{t,|f|}(x)$ with $\varphi_{t,|f|}$ defined in Lemma~\ref{error} b), therefore $\psi_t(x) \to 0$ as $t \searrow 0$. For the first term of the right hand side of the equality, we have, using
$$
\mathbb{E}_x\left[1_{\{N_t(\mathbb{X}) = 0\}} \right]=\mathbb{P}_x\left({\{N_t(\mathbb{X}) = 0\}} \right)=\exp(-t\mathrm{deg}(x)) ,
$$
the convergence

\begin{align*}
\lefteqn{\frac{\mathbb{E}_x\left[1_{\{N_t(\mathbb{X}) = 0\}} \exp(\IAA_t(v,\theta|\mathbb{X})) f(x)\right] -f(x)}{t}}\\
&=\frac{\mathbb{E}_x\left[1_{\{N_t(\mathbb{X}) = 0\}}\exp\big(-t i( v(x)+\mathrm{deg}(x))+t\mathrm{deg}(x)\big) f(x)\right] -f(x)}{t}\\
&\to -i( v(x)+\mathrm{deg}(x)f(x).
\end{align*}
as $t\searrow 0$. Turning to the second term of the right hand side of the equation, setting
$$
 a(y):=-i(v(y)+\mathrm{deg}(y))+\mathrm{deg}(y) ,\quad  y\in W ,
$$
we obtain

\begin{align*}
\mathbb{E}_x&\left[1_{\{N_t(\mathbb{X}) = 1,\mathbb{X}_{\tau_{1}}\in W\}}i \exp(\IAA_t(v,\theta|\mathbb{X})) f(\mathbb{X}_t)\right]\\
&=i\sum_{y \in W}\exp(i \theta(x,y))f(y) \underbrace{\mathbb{E}_x\left[1_{\{N_t(\mathbb{X}) = 1,\mathbb{X}_{\tau_1}= y\}} \exp\Big(-\tau_1a(x) - (t-\tau_1)a(y)\Big) \right]}_{=:\rho_t(x,y)}.
\end{align*}
Setting
$$
C := 2\max \{\mathrm{deg}(x)\mid x\in W\}
$$
and using $\tau_1 \leq t$ on $\{N_t(\mathbb{X}) = 1\}$, we get
\[
\exp(-tC)\mathbb{P}_x(N_t(\mathbb{X}) = 1,\mathbb{X}_{\tau_1}= y) \leq| \rho_t(x,y)|\leq \exp(tC)\mathbb{P}_x(N_t(\mathbb{X}) = 1,\mathbb{X}_{\tau_1}= y).
\]
Since by Lemma~\ref{error}~(a)~(\ref{inti5})
$$
 \mathbb{P}_x(N_t(\mathbb{X}) = 1,\mathbb{X}_{\tau_1}= y)/t\to b(x,y)/m(x) 
$$
 this shows that 
$$
\rho_t(x,y)/t\to b(x,y)/m(x)
$$
 as $t\searrow 0$.  As $W$ is finite, we conclude

$$\frac{1}{t}\mathbb{E}_x\left[ 1_{\{N_t(\mathbb{X}) = 1,\mathbb{X}_{\tau_{1}}\in W\}}i \exp(\IAA_t(v,\theta|\mathbb{X})) f(\mathbb{X}_t) \right] \longrightarrow\frac{i}{m(x)}\sum_{y \in W} b(x,y)\exp(i \theta(x,y))f(y)\quad\text{as $t\searrow0$},   $$
so,  we infer
$$
\frac{U_t(v,\theta,W)f(x)-f(x)}{t} \longrightarrow - iL^{(W)}_{v,\theta}f(x)\>\>\text{ as $t\searrow0$.}
$$
\end{proof}

With these preparations we can now prove Theorem~\ref{p:finite}.

\begin{proof}[Proof of Proposition~\ref{p:finite}] For finite $W\subseteq X$, we have $\ell^{2}(W,m)=C_{c}(W)$. In particular, $L_{v,\theta}^{(W)}$ is a finite dimensional operator and the convergence
\[
-iL_{v,\theta}^{(W)}=\lim_{t\searrow 0}\frac{1}{t}\left(U_t(v,\theta,W)-\mathrm{id}\right)
\]
from Lemma~\ref{generator} holds in the $\ell^{2}(W,m)$ sense. Therefore, the generator of the strongly continuous semigroup  $(U_t(v,\theta,W))_{t\geq 0}$ is given by $L_{v,\theta}^{(W)}$. It follows that $\exp(-itL_{v,\theta}^{(W)})=U_t(v,\theta,W)$ for all $t\ge0$.
\end{proof}

\subsection{Proof of Theorem \ref{main} in the general case}

For any subset $W\subseteq X$ we have a canonically given inclusion operator
$$
\iota_W: \ell^2(W,m)\hookrightarrow \ell^2(X,m),
$$
which comes from extending functions to zero away from $W$, and its adjoint will be denoted with $\pi_W:=\iota^*_W$. Note that $\pi_W$ is given by the restriction map $f\mapsto f|_W$. The following geometric approximation is based on the Mosco convergence of the quadratic forms and will allow us to extend the Feynman path integral from finite to arbitrary graphs:

\begin{prop} \label{approx1}
Suppose $Q^{(c)}_{v,\theta}$ is semi-bounded and closable and let $(X_n)_{n\in\IN}$ be an exhausting sequence for  $X$, that is,  $X_{n}\subseteq X_{n+1}$ for all $n$ and  $X=\bigcup_{n\in\IN}X_{n}$. Then, for all $t\geq 0$, one as
$$
\iota_{X_n} \exp(-itL_{v,\theta}^{(X_n)})\pi_{X_n} \to \exp(-itL_{v,\theta}) \text{ strongly in } \ell^2(X,m) \mbox{ as }n\to\infty.$$
\end{prop}
\begin{proof} By Theorem~\ref{mosco.char} it suffices to show that the forms $Q^{(X_n)}_{v,\theta}$ converge to $Q_{v,\theta}$ as $n\to\infty$ in the generalized Mosco sense. But this has been shown in \cite{GKS}.
 \end{proof}

\begin{proof}[Proof of Theorem~\ref{main}] Let $f\in\ell^2(X,m)$, $x\in X$, $t>0$. We are going to prove
$$
\exp(-itL_{v,\theta})f (x) = \mathbb{E}_x\left[1_{\{t<\tau, N_t(\mathbb{X})<\infty\}} i^{N_t(\mathbb{X})}\exp(\IAA_t(v,\theta|\mathbb{X})) f(\mathbb{X}_t)\right].
$$
In view of (\ref{inti2}), this clearly implies
\begin{align*}
&\mathbb{E}_x\left[1_{\{t<\tau, N_t(\mathbb{X})<\infty\}} i^{N_t(\mathbb{X})}\exp(\IAA_t(v,\theta|\mathbb{X})) f(\mathbb{X}_t)\right]\\
&=\sum_{y\in X} \mathbb{E}_x\left[1_{\{\IX_t=y, N_t(\mathbb{X})<\infty\}} i^{N_t(\mathbb{X})}\exp(\IAA_t(v,\theta|\mathbb{X})) f(\mathbb{X}_t)\right],
\end{align*}
proving the asserted formula. Let $(X_n)$ be an exhausting sequence for $X$. Then, Proposition~\ref{approx1} implies the pointwise convergence
$$
\exp(-itL_{v,\theta})f (x) = \lim_{n \to \infty} \iota_{X_n} \exp(-itL^{(X_n)}_{v,\theta})\pi_{X_n} f (x).
$$
Combining this with Proposition~\ref{p:finite}, it remains to prove the equation
\begin{align*}
&\lim_{n\to \infty} \mathbb{E}_x\left[1_{\{t<\tau_{X_n}, N_t(\mathbb{X})<\infty\}} i^{N_t(\mathbb{X})}\exp(\IAA_t(v,\theta|\mathbb{X}) )\pi_{X_n} f(\mathbb{X}_t)\right] \\
&= \mathbb{E}_x\left[1_{\{t<\tau, N_t(\mathbb{X})<\infty\}} i^{N_t(\mathbb{X})}\exp(\IAA_t(v,\theta|\mathbb{X})) f(\mathbb{X}_t)\right]. 
\end{align*}
This however follows from dominated convergence, as we have
$$
\left|1_{\{t<\tau_{X_n}, N_t(\mathbb{X})<\infty\}} i^{N_t(\mathbb{X})}\exp(\IAA_t(v,\theta|\mathbb{X}))\pi_{X_n} f(\mathbb{X}_t)\right|\leq 1_{\{t<\tau_{X }, N_t(\mathbb{X})<\infty\}} \exp(\int^t_0 \mathrm{deg}(\mathbb{X}_s) ds) | f(\mathbb{X}_t)|
$$
and 
$$
\mathbb{E}_x\left[1_{\{t<\tau_{X}\}} \exp\left(\int^t_0 \mathrm{deg}(\mathbb{X}_s) ds\right)| f(\mathbb{X}_t)|\right]
=\exp(-t L_{-\mathrm{deg},0}   )|f|(x)<\infty
$$
by the usual Feynman-Kac formula \cite{GKS}.
\end{proof}

\appendix

\section{Mosco-convergence}

Let $\IHH_k $, $k \in\IN$, and $ \IHH $ be Hilbert spaces. Suppose $q_k$ and $q$ are densely defined closed symmetric sesquilinear forms on $\IHH_k$ and $\IHH$, respectively, which are bounded below by a constant $c> -\infty$ which is \emph{uniform} in $k$. Each $q_k$ is understood to be defined on the whole space $\IHH_k$ by the convention $q_k(u) = \infty$ whenever $u \in \IHH_k \setminus \mathrm{Dom}(q_k)$. Furthermore, we suppose that there exist bounded operators $\iota_k:\IHH_k \to \IHH$ such that $\pi_k := \iota_k ^*$ is a left inverse of $\iota_k$, that is
$$\as{\pi_kf, f_k}  = \as{f,\iota_kf_k}\text{ and } \pi_k \iota_kf_k = f_k, \text{ for all } f\in \IHH, f_k\in \IHH_k.$$
Moreover, we assume that $\pi_k$ satisfies
$$\sup_{k\in\IN}\|\pi_k\|< \infty \text{ and } \lim_{k\to \infty} \|\pi_kf\|  = \|f\|.$$

\begin{definition} \label{mosco}
In the above situation, we say that $q_k$ is \emph{Mosco convergent} to $q$ as $k\to\infty$ \emph{in the generalized sense}, if the following conditions hold:
\begin{itemize}
\item[(a)] If $u_k \in \IHH_k$, $u \in \IHH$ and $\iota_ku_k \to u$ weakly in $\IHH$, then
$$\liminf_{k \to \infty}\left(q_k(u_k) + c\|u_k\|_k^2\right) \geq q(u) + c\|u\|^2.$$
\item[(b)] For every $u \in \IHH$ there exist $u_k \in \IHH_k$, such that $\iota_k u_k \to u$ in $\IHH$ and
$$\limsup_{k \to \infty}\left(q_k(u_k) + c \|u_k\|^2\right) \leq q(u) + c \|u\|^2.$$
\end{itemize}
\end{definition}

We  denote by $L_k$ the self-adjoint operator corresponding to $q_k$ and let $L$ be the self-adjoint operator corresponding to $q$ which are both bounded from below by $ c $.  We will need the following generalization of the characterization of Mosco convergence from \cite{CKK} (see also the appendix of \cite{GKS}). Given an interval $I\subset \IR$, we denote by 
$ C_{b}(I) $ the bounded continuous functions on $ I $ and by $ C_{\infty}(I) $ the space of continuous functions that become arbitrarily small outside of every compact set of $I$.

\begin{thm} \label{mosco.char}
If $q_k$ is Mosco convergent to $q$ as $k\to\infty$ in the generalized sense, then one has $\iota_k \psi(L_k)\pi_k \to \psi(L)$ as $k\to\infty$ strongly for every $\psi\in C_b(\IR)$.

\end{thm}

\begin{remark} As the proof below shows, the fact that one can take $C_{\infty}(\IR)$ in the above statement is a rather simple consequence of known results, the Stone-Weierstrass Theorem and the spectral calculus. The point of Theorem \ref{mosco.char} is that one can even take $C_b(\IR)$, which plays an essential role in this paper, for we are interested in operators of the form $\exp(-itL)$.
\end{remark}

\begin{proof} It is well-known that Mosco convergence is equivalent to 
\begin{align}\label{popo}
\iota_k \exp(-aL_k)\pi_k \to \exp(-aL)\quad\text{ as $k\to\infty$ },
\end{align}
strongly and locally uniformly in $a\geq 0$ (cf. Theorem 3.8 in \cite{CKK} and the appendix of \cite{GKS}). We are going to prove that the latter semigroup convergence implies $\iota_k \psi(L_k)\pi_k \to \psi(L)$ as $k\to\infty$ strongly for every $\psi\in C_b([c,\infty))$, proving the claim, as the spectra of $L$ and $L_k$ are subsets of $[c,\infty)$. To this end, we are going to follow the proof of Theorem VIII.20 in \cite{RS1} (which treats the case $\IHH_k=\IHH$ and $\pi_k=\mathrm{id}_{\IHH}$). \vspace{1mm}

Step 1: The claim holds for all $\psi\in C_{\infty}([c,\infty))$.\\
Proof: Let us denote with $\IAA$ the space of complex linear combinations of functions of the form $x\mapsto \exp(-a x)$ on $[c,\infty)$, where $a,b\geq 0$. Then $\IAA$ is a separating unital *-subalgebra of $C_{\infty}([c,\infty))$, thus dense in $C_{\infty}([c,\infty))$ by Stone-Weierstrass. Given an arbitrary $\varepsilon >0$, and $\psi\in C_{\infty}([c,\infty))$ we thus find $\psi_{\varepsilon }\in \IAA$ with $\left\|\psi-\psi_{\varepsilon}\right\|_{\infty}<\varepsilon $. By the spectral calculus we have
\begin{align}\label{kla}
\left\|\psi(L_k)-\psi_{\varepsilon}(L_k)\right\| <\varepsilon ,\quad \left\|\psi(L )-\psi_{\varepsilon}(L )\right\|<\varepsilon .
\end{align}
for the operator norms. Let now $f\in \IHH$. Then we can estimate as follows,
\begin{align*}
&\left\| \big(\iota_k \psi(L_k)\pi_k - \psi(L)\big)f\right\|\\
&=\left\| \Big(\iota_k \psi(L_k)\pi_k - \psi(L)   + \iota_k \psi_{\varepsilon}(L_k)\pi_k-\iota_k\psi_{\varepsilon}(L_k)\pi_k +\psi_{\varepsilon}(L)- \psi_{\varepsilon}(L)\Big)f\right\|\\
&\leq \left\| \iota_k \psi(L_k)\pi_k f-\iota_k \psi_{\varepsilon}(L_k)\pi_kf \right\|+\left\|  \psi_{\varepsilon}(L)f- \psi(L)f  \right\|+\left\|     \iota_k \psi_{\varepsilon}(L_k)\pi_kf- \psi_{\varepsilon}(L)f\right\|.
\end{align*}
The last summand is $<\varepsilon$ for large $k$ by (\ref{popo}) (which clearly extends from exponentials to $\IAA$), the second summand is $<\varepsilon\left\|f\right\|$ for all $k$ by (\ref{kla}), and the first summand is 
$$
< \sup_k\left\|\pi_k\right\|^2 \left\|f\right\|\varepsilon
$$
for all $k$ by (\ref{kla}), completing the proof of Step 1.\vspace{1mm}

Step 2: The claim holds for all $\psi\in C_{b}([c,\infty))$.\\
Proof: Fix $f\in \IHH$ and $\varepsilon>0$. For every $l\in \IN$ set $g_l(x):=\mathrm{e}^{-x^2/l}$, a function in $C_{\infty}([c,\infty))$. Since $g_l(x)\to 1$ from below, the spectral calculus implies $g_l(B)\to \mathrm{id}$ strongly as $l\to\infty$ for every self-adjoint operator $B$. We can thus fix an $l$ such that  
$$
\left\|f-g_l(L)f\right\|<\varepsilon.
$$
Furthermore, let us set
$$
C_1:=\max\Big(\left\|\psi(L)\right\|,\sup_k\left\|\psi(L_k)\right\|\Big)\leq \left\|\psi\right\|_{\infty},\quad C_2:=\sup_k \left\|\pi_k\right\|=\sup_k \left\|\iota_k\right\|.
$$
Then for large $k$ we can estimate as follows: 
\begin{align*}
\lefteqn{\left\|\iota_k\psi(L_k)\pi_kf-\psi(L)f\right\|}\\
&\leq \left\|\psi(L) g_l(L)f -\psi(L)f\right\|+\left\|\iota_k\psi(L_k)g_l(L_k)\pi_kf- \psi(L) g_l(L)f\right\|\\
&\quad\>\>+\left\| \iota_k\psi(L_k)\pi_kf-\iota_k\psi(L_k)g_l(L_k)\pi_kf\right\|\\
&\leq C_1\varepsilon+\varepsilon+C_{1}C_2\left\| \pi_kf-g_l(L_k)\pi_kf\right\|\\
& \leq C_1\varepsilon+\varepsilon+C_{1}C_2\left\| \pi_kf-\pi_kg_l(L)f\right\|+C_{1}C_{2}\left\|g_l(L_k)\pi_kf+\pi_kg_l(L)f\right\|\\
&\leq  C_1\varepsilon+\varepsilon+C_{1}C_2^2\left\| f-g_l(L)f\right\|+C_{1}C_2\left\|\pi_k\iota_kg_l(L_k)\pi_kf -\pi_kg_l(L)f\right\|\\
&\leq  C_1\varepsilon+\varepsilon+C_{1}C_2^2\varepsilon+C_{1}C_2^2\left\|\iota_kg_l(L_k)\pi_kf -g_l(L)f\right\|\\
&\leq  C_1\varepsilon+\varepsilon+C_{1}C_2^2\varepsilon+C_{1}C_2^2\varepsilon,
\end{align*}
where we have used $\psi g_l\in C_{\infty}([c,\infty))$ and step 1 for the second step, $\pi_k\iota_k=\mathrm{id}_{\IHH_k}$ for the fifth step, and $  g_l\in C_{\infty}([c,\infty))$ and Step 1 for the last step. This completes the proof.

\end{proof}

{\bf Acknowledgements:} The authors would like to thank Burkhard Eden, Evgeny Korotyaev, Ognjen Milatovic and Matthias Staudacher for very helpful discussions.

\end{document}